\documentclass{amsart}
\usepackage{amssymb,latexsym}

\theoremstyle{plain}
\newtheorem{theorem}{Theorem}

\newtheorem{lemma}{Lemma}

\newtheorem{remark}{Remark}
\newtheorem{convention}{Convention}

\theoremstyle{definition}
\newtheorem{definition}{Definition}

\theoremstyle{remark}

\begin{document}

% Two authors
\title[an Extension from an Algebra to a $\sigma$-Algebra]{A realization of measurable sets as limit points }
\author{Jun Tanaka}
\address{University of California, Riverside, USA}
\email{juntanaka@math.ucr.edu, yonigeninnin@gmail.com, junextension@hotmail.com}

\author{Peter F. Mcloughlin}
\address{University of California, Riverside, USA}
\email{pmcloughlin@aol.com}

% End two authors

\keywords{Measure Theory, Caratheodory Extension Theorem, Metric}
\subjclass[2000]{Primary: 28A12, 28B20}
\date{December, 10, 2007}

\begin{abstract}
  Starting with a $\sigma$-finite measure on an algebra, we define a pseudometric and show how measurable sets from the Caratheodory Extension Theorem can be thought of as limit points of Cauchy sequences in the algebra.
\end{abstract}
\maketitle

\section{Introduction}\label{S:intro}

Generations of graduate students in the mathematical sciences have wrestled to understand the Caratheodory Extension Theorem. A crucial part of any advanced Real Analysis course, the process extends an algebra to a $\sigma$-algebra and a measure on an algebra to a measure on a $\sigma$-algebra. This is a powerful tool in measure theory and statistics.

Let $\mu$ be a measure on an algebra $\Omega \subset$ $ \mathbf{P}(X)$. The Caratheodory Extension consists of the following two main definitions: Definition 1; $\mu^{\ast}$ : $\mathbf{P}(X)   \rightarrow$ $\mathbb{R}^{+}$ where $\mu^{\ast} (E)$ = $\inf \{ \sum \mu (A_{i}) \mid  E  \subseteq \cup  A_{i}$ and $A_{i}  \in \Omega$ for all i $\geq $ 1$\}$ [1].
Definition 2; E $\in \mathbf{P}(X)$ is Measurable with respect to $\mu^{\ast}$ if for every set $A$ $\in \mathbf{P}(X)$, we have $\mu^{\ast} (A)= \mu^{\ast}(A \cap E) + \mu^{\ast}(A \cap E^{\textbf{C}} )$.

In general $\mu^{\ast}$ is not a measure on $\mathbf{P}(X)$. In fact it turns out that the set $M$ of measurable sets forms a $\sigma$-algebra and that ($\mu^{\ast}$, $M$) is a measure space. Where definition 2 came from always seemed a mystery to the authors.

In this paper, we propose what we believe to be a simplification of the proof of the Caratheodory Extension Theorem when $\mu$ is $\sigma$-finite. Our first step is to define the pseudometric $d (\cdot,\cdot)$ = $\mu^{\ast} (\cdot \triangle \cdot)$. We will say a sequence $\{ B_{n} \}$ in $\Omega$ is $\mu$-Cauchy iff $\lim d ( B_{n}  , B_{m}   )  \rightarrow $ 0. Next we define $ \widetilde{\mathbf{S}}$ = $\{S \in \mathbf{P}(X)  | \exists \ \mu$-Cauchy sequence  $\{B_{n}   \}   $ s.t. $ \lim \mu^{\ast} ( B_{n}   \triangle S)  =0   \}$.

In particular, $ \widetilde{\mathbf{S}}$ merely consists of limit points of $\mu$-Cauchy sequences. The authors believe that $ \widetilde{\mathbf{S}}$ is both an intuitive and natural definition. Our $ \widetilde{\mathbf{S}}$ will replace Definition 2 in the Caratheodory Extension Theorem. For $S$ $\in \widetilde{\mathbf{S}}$, we define $\widetilde{\mu} (S) $ =  $\lim \mu ( B_{n}   )   $ where $\{ B_{n}   \}$ is a $\mu$-Cauchy sequence which converges to $S$. We will show that ($\widetilde{\mu}$, $ \widetilde{\mathbf{S}}$) is the measure space obtained by the Caratheodory Extension Theorem.

\section{Main Results}
Let $\mu$ be a finite measure on an algebra $\Omega \subset$ $ \mathbf{P}(X)$ and $\mu^{\ast}$ be the outer measure defined by $\mu^{\ast}(A)$  = $\inf \{ \sum \mu (A_{i}) \mid  E  \subseteq \cup A_{i}$ and $A_{i}  \in \Omega$ for all i $\geq $ 1$\}$ for any $A \in \mathbf{P}(X)$.

\begin{lemma}\label{L:1}
$d$ : $\mathbf{P}(X)  \times   \mathbf{P}(X)   \rightarrow$ $\mathbb{R}^{+}$ where $d(A,B)$ = $\mu^{\ast}(A \triangle B)$ is a pseudometric.
\end{lemma}

\begin{remark}\label{R:1}
$d$ is a metric on $\mathbf{P}(X)_{\diagup_{\sim} } $ where A $\sim$ B  iff   $\mu^{\ast}(A \triangle B)$ =0, [2].
\end{remark}

\begin{lemma}\label{L:2}
$d(A_{1} \cup A_{2},A_{3} \cup A_{4})$  $\leq$   $d(A_{1} , A_{3})$ +  $d(A_{2} , A_{4})$ for any $A_{i} \in  $ $\mathbf{P}(X)$ .
\end{lemma}

\begin{proof}

$d(A_{1} \cup A_{2},A_{3} \cup A_{4})$

= $\mu^{\ast} ((A_{1} \cup A_{2}) \cap (A_{3} \cup A_{4})^{\textbf{C}}  \ \cup \  (A_{3} \cup A_{4}) \cup (A_{1} \cup A_{2})^{\textbf{C}} )$

= $\mu^{\ast} (A_{1} \cap A_{3}^{\textbf{C}} \cap A_{4}^{\textbf{C}} \ \ \cup \ \ A_{2} \cap A_{3}^{\textbf{C}}  \cap A_{4}^{\textbf{C}} \ \ \cup \ \  A_{3} \cap A_{1}^{\textbf{C}} \cap A_{2}^{\textbf{C}} \ \  \cup \ \  A_{4} \cap A_{1}^{\textbf{C}}  \cap A_{2}^{\textbf{C}})$

$\leq$ $\mu^{\ast} (A_{1} \cap A_{3}^{\textbf{C}}  \ \cup  \  A_{2} \cap A_{4}^{\textbf{C}} \ \cup \ A_{3} \cap A_{1}^{\textbf{C}} \  \cup \ A_{4} \cap A_{2}^{\textbf{C}})$

$\leq$ $\mu^{\ast}(A_{1} \triangle A_{3})$ + $\mu^{\ast}(A_{2} \triangle A_{4})$ = $d(A_{1},A_{3})$ + $d(A_{2},A_{4})$.

\end{proof}

\begin{lemma}\label{L:3}
$\mu^{\ast} |_{\Omega} = \mu $, from page 292 of [1].
\end{lemma}

\begin{definition}\label{D:1}
If $B_{n} \in \Omega$, then $\{B_{n} \}$ is called a $\mu  $-Cauchy sequence if $\lim \mu(B_{n} \triangle B_{m})$ $\rightarrow$ 0 as $n,m \rightarrow \infty$.
\end{definition}

\begin{convention} Unless otherwise stated, $\{B^{\alpha}_{n} \}$ and $\{B_{n}   \}   $ will be $\mu  $-Cauchy sequences.
\end{convention}

\begin{lemma}\label{L:limit} $\{B_{n}   \}   $ a $\mu  $-Cauchy sequence implies $\{ \mu ( B_{n} ) \}$ is a Cauchy sequence of real numbers.
\end{lemma}

\begin{proof}
$ |  \mu(B_{m} ) - \mu ( B_{n} )  | $ $ \leq $  $ \mu ( B_{m} \triangle B_{m} )  $ $\rightarrow  $ 0. Thus, $\{ \mu ( B_{m}) \}$ is a Cauchy sequence in $\mathbb{R}$. Thus, the claim follows.

\end{proof}

\begin{lemma}\label{L:4}

$\{B^{\alpha}_{n} \cup  B^{\gamma}_{n} \}$ and $\{(B^{\alpha}_{n})^{\textbf{C}} \}$ are  $\mu  $-Cauchy sequences.

\end{lemma}

\begin{proof} by Lemma $\ref{L:2}$,

$d(B^{\alpha}_{n} \cup B^{\gamma}_{n},B^{\alpha}_{m} \cup B^{\gamma}_{m})$ $\leq$ $d(B^{\alpha}_{n} ,B^{\alpha}_{m} )$ + $d( B^{\gamma}_{n}, B^{\gamma}_{m})$.

Moreover, $d((B^{\alpha}_{n})^{\textbf{C}},(B^{\alpha}_{m})^{\textbf{C}})$ =  $\mu((B^{\alpha}_{n})^{\textbf{C}} \triangle (B^{\alpha}_{m})^{\textbf{C}})$

=  $\mu(B^{\alpha}_{n} \triangle B^{\alpha}_{m})$ =  $d (B^{\alpha}_{n}, B^{\alpha}_{m})$. Hence the claim follows.

\end{proof}

\begin{definition}

Let $ \widetilde{\mathbf{S}}$ = $\{S \in \mathbf{P}(X)  | \exists \ \mu$-Cauchy sequence  $\{B_{n}   \}   $ s.t. $ \lim \mu^{\ast} ( B_{n}   \triangle S)  =0   \}$

\end{definition}

\begin{lemma}\label{L:5} $ \widetilde{\mathbf{S}}$ is an algebra.
\end{lemma}

\begin{proof}
Let $S_{1} , S_{2} \in$ $ \widetilde{\mathbf{S}}$. There exist $\mu$-Cauchy sequences  $ \{ B^{1}_{n} \}$ and $ \{ B^{2}_{n} \}$ which correspond to $S_{1} $ and $ S_{2}$ respectively.

Now $ \mu^{\ast} (  S_{1} \cup S_{2}  \  \triangle \   B^{1}_{n} \cup  B^{2}_{n}    )  \leq  $ $  \mu^{\ast} (  S_{1}   \Delta    B^{1}_{n}     )      $   +  $\mu^{\ast} (  S_{2}   \Delta    B^{2}_{n}    )   $ by Lemma $\ref{L:2}$.

Thus , since $\{ B^{1}_{n} \cup  B^{2}_{n} \}$ is a $\mu$-Cauchy sequence by Lemma $\ref{L:4}$, by taking the limit on both sides,  $S_{1} \cup S_{2} \in$ $ \widetilde{\mathbf{S}}$.

In addition, this covers the finite union case. Note now that $\{ B^{1}_{n} \cup  B^{2}_{n} \}$ is a $\mu $-Cauchy sequence which corresponds to  $S_{1} \cup S_{2}$.

Similarly,  $(S_{1})^{\textbf{C}} \in$ $ \widetilde{\mathbf{S}}$, thus $ \widetilde{\mathbf{S}}$ is an algebra.

\end{proof}

\begin{definition}\label{D:2} Definition of the measure on $ \widetilde{\mathbf{S}}$.

Let $S \in$ $ \widetilde{\mathbf{S}}$. Then there exists a $\mu$-Cauchy sequence $ \{ B_{n} \}$ which corresponds to $S$. By lemma $\ref{L:limit}$, $\lim \mu ( B_{n}   )   $ exists. We define $\widetilde{\mu} (S) $ =  $\lim \mu ( B_{n}   )   $.

\begin{remark}\label{R:2} By lemmas $\ref{L:1}$ and $\ref{L:3}$, $ |  \mu^{\ast} (S ) - \mu ( B_{n} )  | $=$ |  \mu^{\ast} (S \triangle \emptyset) - \mu^{\ast} ( B_{n} \triangle \emptyset)  |   \leq $  $ \mu^{\ast} ( B_{n} \triangle S )  $ for $S$ in $ \widetilde{\mathbf{S}}$. Thus, $\widetilde{\mu} (S) $= $\lim \mu ( B_{n}   )   $ =$ \mu^{\ast} (S ) $. Hence, $\mu^{\ast}  |_{\widetilde{\mathbf{S}}} $ = $\widetilde{\mu}$.

\end{remark}
\end{definition}

\begin{lemma}\label{L:6}  If $S_{1} , S_{2}  \in$ $ \widetilde{\mathbf{S}}$ are disjoint, then $\widetilde{\mu} (S_{1} \cup S_{2} ) $ = $ \widetilde{\mu} (S_{1} )   $ + $ \widetilde{\mu} (S_{2} )   $.
\end{lemma}

\begin{proof}
Let $S_{1} , S_{2}  \in$ $ \widetilde{\mathbf{S}}$ be disjoint, then there exists $\mu$-Cauchy sequences  $ \{ B^{1}_{n} \} $ and $ \{ B^{2}_{n} \}$ which correspond to $S_{1} $ and $ S_{2}$ respectively. By Lemma $\ref{L:5}$, $ S_{1}^{\textbf{C}} \cup  S_{2}^{\textbf{C}} $ $\in \widetilde{\mathbf{S}}$ and $ \{ (B^{1}_{n})^{\textbf{C}} \cup (B^{2}_{n})^{\textbf{C}}  \} $ is a $\mu$-Cauchy sequence which corresponds to $ S_{1}^{\textbf{C}} \cup  S_{2}^{\textbf{C}} $. It follows $\lim \mu^{\ast} ( S_{1}^{\textbf{C}} \cup  S_{2}^{\textbf{C}}   \  \triangle \ (B_{n}^{1} )^{\textbf{C}} \cup   (B_{n}^{2} )^{\textbf{C}}     ) $ = 0. Note ($S_{1}^{\textbf{C}} \cup  S_{2}^{\textbf{C}}$ ) $\triangle  $  ($(B_{n}^{1} )^{\textbf{C}} \cup   (B_{n}^{2} )^{\textbf{C}}    $)= $ B_{n}^{1} \cap B_{n}^{2}  $ since $S_{1}$ and $S_{2}$ are disjoint. Now we have   $\lim \mu^{\ast} (  B_{n}^{1} \cap B_{n}^{2} ) $ = $\lim \mu (  B_{n}^{1} \cap B_{n}^{2} ) $ = 0   $\ $  (1).

Note $B^{1}_{n}$ =  $(B^{1}_{n}  \cap    B^{2}_{n})$ $\cup$  $B^{1}_{n} \cap  (B^{2}_{n})^{\textbf{C}} $  implies $\mu(B^{1}_{n})$ = $\mu( B^{1}_{n}  \cap    B^{2}_{n}) $ + $\mu( B^{1}_{n} \cap  (B^{2}_{n})^{\textbf{C}} )$ and $\lim   \mu(B^{1}_{n})$ = $\lim  \mu( B^{1}_{n} \cap  (B^{2}_{n})^{\textbf{C}} )$ by (1). Similarly, $\lim   \mu(B^{2}_{n})$ = $\lim  \mu( B^{2}_{n} \cap  (B^{1}_{n})^{\textbf{C}} )$.

Moreover, $\mu( B^{1}_{n}  \cup    B^{2}_{n})$ = $\mu( B^{1}_{n}  \cap    B^{2}_{n} \ \  \cup \ \    B^{1}_{n} \cap  (B^{2}_{n})^{\textbf{C}}    \ \ \cup  \ \  B^{2}_{n} \cap  (B^{1}_{n})^{\textbf{C}}  )$

= $\mu( B^{1}_{n}  \cap    B^{2}_{n} )$ +    $\mu (B^{1}_{n} \cap  (B^{2}_{n})^{\textbf{C}}) $ +  $\mu (B^{2}_{n} \cap  (B^{1}_{n})^{\textbf{C}}  )$. Thus, $\widetilde{\mu} (S_{1} \cup S_{2} ) $ = $\lim \mu (    B^{1}_{n}   \cup  B^{2}_{n}     )$ = $\lim \mu (    B^{1}_{n}) $  +   $\lim \mu (   B^{2}_{n}     )$ = $ \widetilde{\mu} (S_{1} )   $ + $ \widetilde{\mu} (S_{2} )  $ .

Therefore, $ \widetilde{\mu} $ is finitely additive.

\end{proof}

\begin{lemma}\label{L:9} $ \widetilde{\mathbf{S}}$ is a $\sigma$-algebra.
\end{lemma}

\begin{proof}
Let $S_{i} \in$ $ \widetilde{\mathbf{S}}$ be mutually disjoint and $\{B^{i}_{n}\} $ correspond to each $S_{i}$, for all i $ \geq$ 1. By lemma $\ref{L:6}$ and remark $\ref{R:2}$, we have $\sum_{i=1}^{n} \mu^{\ast}(S_{i}) = \sum_{i=1}^{n} \widetilde{\mu} (S_{i})= \widetilde{\mu} ( \cup_{i=1}^{n} S_{i})$= $\mu^{\ast}( \cup_{i=1}^{n}   S_{i} )   $.
We will show that $\cup_{i=1}^{\infty}   S_{i}$ $\in \widetilde{\mathbf{S}}$. First we will show that
\[
\mu^{\ast}( \cup_{i=1}^{\infty}   S_{i} ) \leq \sum_{i=1}^{\infty} \mu^{\ast}(S_{i})  <  \infty .
\]

For any $n$, $\mu^{\ast}( \cup_{i=1}^{\infty}   S_{i} ) \leq $ $\mu^{\ast}( \cup_{i=n+1}^{\infty}   S_{i} ) $ + $\mu^{\ast}( \cup_{i=1}^{n}   S_{i} ) \leq $ $\mu^{\ast}( X ) + \mu^{\ast}( X )$ by the subadditivity of $\mu^{\ast}$.
Thus, $\mu^{\ast}( \cup_{i=n+1}^{\infty}   S_{i} ) $ + $\sum_{i=1}^{n} \mu^{\ast}(S_{i}) \leq $ 2$\mu^{\ast}( X )$ = 2$\mu (X)$ for any n. By the finiteness of $\mu$ and the subadditivity of $\mu^{\ast}$, we must have $\mu^{\ast}( \cup_{i=1}^{\infty}   S_{i} )$ $\leq$  $\sum_{i=1}^{\infty} \mu^{\ast}(S_{i}) \leq$ 2$\mu( X ) <  \infty$ $\  $ (2).

Now we are going to construct a $\mu$-Cauchy sequence which converges to $\cup_{i=1}^{\infty}   S_{i}$ with respect to the pseudometric $d (\cdot,\cdot)$ = $\mu^{\ast} (\cdot \triangle \cdot)$.

\[\begin{array}{cccccc}
  B^{1}_{1} & B^{1}_{2} & \cdots & B^{1}_{k}  & \cdots & \longrightarrow S_{1} \\
  B^{1}_{1} \cup B^{2}_{1} & B^{1}_{2} \cup B^{2}_{2} & \cdots & B^{1}_{k} \cup B^{2}_{k} & \cdots & \longrightarrow S_{1} \cup S_{2} \\
  \vdots & \vdots & \vdots & \vdots & \vdots & \vdots   \\
  B^{1}_{1} \cup B^{2}_{1} \cdots \cup B^{m}_{1} & B^{1}_{2} \cup B^{2}_{2} \cdots \cup B^{m}_{2} & \cdots  & B^{1}_{k} \cup B^{2}_{k} \cdots \cup B^{m}_{k} & \cdots & \longrightarrow S_{1} \cup S_{2} \cup \cdots \cup S_{m} \\
  \vdots & \vdots & \vdots & \vdots & \vdots & \vdots
\end{array}\]

Note that for each $S_{1} \cup S_{2} \cup \cdots \cup S_{m}$, we can choose a $K$ so that $B^{1}_{K} \cup B^{2}_{K} \cdots B^{m}_{K}$ is arbitrary close to it in terms of the pseudometric $d (\cdot,\cdot)$.

Now, (2) implies for any $L$ there exists an $N_{L}$ such that $\mu^{\ast}( \cup_{i=N_{L}+1}^{\infty}   S_{i} ) $ $\leq  \sum_{i=N_{L}+1}^{\infty} \mu^{\ast}(S_{i})$  $< \frac{1}{2L}$. Moreover, for each $N_{L}$ there exists an $K_{L}$ such that

$ \mu^{\ast} ( S_{1} \cup S_{2} \cup \cdots \cup S_{N_{L}}    \    \triangle  \  B^{1}_{K_{L}} \cup B^{2}_{K_{L}} \cdots B^{N_{L}}_{K_{L}}   )  < \frac{1}{2L}  $.
Thus, by Lemma $\ref{L:2}$

\[
\mu^{\ast}(\cup_{i=1}^{\infty}   S_{i}  \triangle  \cup_{i=1}^{N_{L}} B^{i}_{K_{L}}     ) \leq  \mu^{\ast}(\cup_{i=N_{L}+1}^{\infty}   S_{i}) + \mu^{\ast}(\cup_{i=1}^{N_{L}}   S_{i}  \triangle  \cup_{i=1}^{N_{L}} B^{i}_{K_{L}}     ) < \frac{1}{L}.
\]

Now, call $Y_{L} $ =  $\cup_{i=1}^{N_{L}} B^{i}_{K_{L}}$ for each $L$. By design, $\{ Y_{L} \}$ converges to $\cup_{i=1}^{\infty}   S_{i} $. Also by the triangle inequality, $\mu ( Y_{n} \triangle Y_{m})$=$\mu^{\ast} (Y_{n} \triangle Y_{m}) \leq$ $\mu^{\ast} (Y_{n} \triangle \cup_{i=1}^{\infty}   S_{i}) + \mu^{\ast} (\cup_{i=1}^{\infty}   S_{i} \triangle Y_{m})  $ implies $\{ Y_{L} \}$ is a $\mu$-Cauchy sequence.

Thus, $\cup_{i=1}^{\infty}   S_{i} \in   \widetilde{\mathbf{S}} $. Therefore, $\widetilde{\mathbf{S}} $ is a $\sigma$-algebra.

\end{proof}

\begin{lemma}\label{L:7} $  \widetilde{\mu}$ is a countably additive measure on $ \widetilde{\mathbf{S}}$.

\end{lemma}

\begin{proof}
By (2) for any $\epsilon > 0$ there exists  an $M $ such that $\widetilde{\mu} ( \cup_{i=n+1}^{\infty}   S_{i} ) \leq$ $ \sum_{i=n+1}^{\infty}   \widetilde{\mu}   ( S_{i} ) < \epsilon $ for $n \geq M  $.

Moreover, $\widetilde{\mu} ( \cup_{i=1}^{\infty}   S_{i} ) = $ $  \widetilde{\mu} ( \cup_{i=n+1}^{\infty}   S_{i} ) + \widetilde{\mu} ( \cup_{i=1}^{n}   S_{i} )   = $ $   \widetilde{\mu} ( \cup_{i=n+1}^{\infty}   S_{i} ) + \sum_{i=1}^{n}   \widetilde{\mu}   ( S_{i} )   $ for all $n$ by Lemmas 6, 7, and 8.

Thus, $\widetilde{\mu} ( \cup_{i=1}^{\infty}   S_{i} )$ =  $\sum_{i=1}^{\infty}   \widetilde{\mu}   ( S_{i} )$. Therefore, $\widetilde{\mu}$ is a countably additive measure on $ \widetilde{\mathbf{S}}$.

\end{proof}

\begin{theorem}\label{T:1} ($\widetilde{\mu}$, $ \widetilde{\mathbf{S}}$) is a measure space.
\end{theorem}

\begin{proof}
The claim follows by Lemmas $\ref{L:9}$ and $\ref{L:7}$.
\end{proof}

\begin{theorem}\label{T:2}
$E$ is a measurable set iff  $E$ is in $ \widetilde{\mathbf{S}}$.
\end{theorem}

\begin{proof}
Suppose now $E$ is measurable. $\mu^{\ast}( E ) < \infty$ implies for any $n$ there exists an $C_{n}=  \cup_{i=1}^{\infty} A^{n}_{i}$ where $A^{n}_{i} \cap A^{n}_{j} = \emptyset $ for i $\neq$ j , $A^{n}_{i} \in \Omega$ for all i $\geq$ 1 such that $E$ $\subseteq $ $C_{n}$ and  $\mu^{\ast}( E )$ $>$ $\mu^{\ast}( C_{n} )$ - $\frac{1}{2n}  $. Now $E$ is measurable implies $\mu^{\ast}( C_{n} )$ = $\mu^{\ast}(C_{n} \cap E)$ +    $\mu^{\ast}( C_{n} \cap E^{\textbf{C}})$ = $\mu^{\ast}(C_{n} \triangle E)$ +  $\mu^{\ast}( E )$ which implies $\mu^{\ast}(C_{n} \triangle E)$ $<$ $\frac{1}{2n}$. Furthermore, $\mu^{\ast}( C_{n} )$ = $\sum_{i=1}^{\infty} \mu^{\ast}( A^{n}_{i}   )$ $<$ $\mu^{\ast}( E )$ + $\frac{1}{2n}$ implies there exists an $N_{n}$ such that  $\sum_{i=N_{n} + 1}^{\infty} \mu^{\ast}( A^{n}_{i}   )$ $<$  $\frac{1}{2n}$ $ \ $ (3). Let $ C_{n}^{N_{n}}  $ = $ \cup_{i=1}^{N_{n}} A^{n}_{i}  $. By the triangle inequality and (3), we have
$\mu^{\ast}(  C_{n}^{N_{n}}   \triangle E)$  $\leq$  $\mu^{\ast}(  C_{n}^{N_{n}}   \triangle  C_{n}  )$ + $\mu^{\ast}(  C_{n}  \triangle E)$ $\leq$ $\frac{1}{n}$. By design  $C_{n}^{N_{n}}$ is in $\Omega$. Moreover, by the triangle inequality for $n > m$, $\mu^{\ast}(  C_{n}^{N_{n}}   \triangle C_{m}^{N_{m}} )$  $\leq$  $\mu^{\ast}(  C_{n}^{N_{n}}   \triangle  E )$ + $\mu^{\ast}(  C_{m}^{N_{m}}    \triangle E)$ $\leq$ $\frac{2}{m}$ implies  $\{ C_{n}^{N_{n}} \} $ is a $\mu $-Cauchy sequence. Hence $E \in$  $ \widetilde{\mathbf{S}}$.\\

Suppose $E \in $ $ \widetilde{\mathbf{S}}$. Then there exists a $\mu$-Cauchy sequence $ \{ B_{n} \}$ such that
$ \lim \mu^{\ast} ( B_{n}   \triangle E)  =0   $. It follows we must have $ \lim \mu^{\ast} ( B_{n}   \cap E^{\textbf{C}} )  =0   $ and $ \lim \mu^{\ast} ( (B_{n}) ^{\textbf{C}}   \cap E)  =0   $ $ \ $ (4). Note for $A \in$ $ \mathbf{P}(X)$ we must have

$\mu^{\ast}( A ) \leq $ $\mu^{\ast}( A \cap E ) $ + $\mu^{\ast}( A \cap E^{\textbf{C}} ) $ by the subadditivity of $\mu^{\ast}$. Moreover, $B_{n}$ is measurable for each $n$ implies $\mu^{\ast}( A )$ = $\mu^{\ast}( A \cap B_{n} ) $ + $\mu^{\ast}( A \cap (B_{n})^{\textbf{C}} ) $, $\mu^{\ast}( A \cap E)$ = $\mu^{\ast}( A \cap E \cap B_{n} ) $ + $\mu^{\ast}( A \cap E \cap (B_{n})^{\textbf{C}} ) $ and $\mu^{\ast}( A \cap E^{\textbf{C}})$ = $\mu^{\ast}( A \cap E^{\textbf{C}} \cap B_{n} ) $ + $\mu^{\ast}( A \cap E^{\textbf{C}} \cap (B_{n})^{\textbf{C}} ) $ $ \ $ (5).

Now by taking limits on both sides of (5) and using (4),

we get $\mu^{\ast}( A )$ = $ \lim \mu^{\ast}( A \cap B_{n} ) $ + $ \lim \mu^{\ast}( A \cap (B_{n})^{\textbf{C}} ) $, $\mu^{\ast}( A \cap E)$ = $\lim \mu^{\ast}( A \cap E \cap B_{n} ) $, and $\mu^{\ast}( A \cap E^{\textbf{C}})$ = $ \lim \mu^{\ast}( A \cap E^{\textbf{C}} \cap (B_{n})^{\textbf{C}} ) $. Moreover, $A \cap E \cap B_{n}$ $\subseteq $ $A \cap B_{n}$ and $A \cap E^{\textbf{C}} \cap (B_{n})^{\textbf{C}}$  $\subseteq $ $A \cap (B_{n})^{\textbf{C}}$ implies

$ \lim \mu^{\ast}( A \cap B_{n} ) $ $\geq$ $\mu^{\ast}( A \cap E)$ and  $ \lim \mu^{\ast}( A \cap (B_{n})^{\textbf{C}} ) $ $\geq$ $\mu^{\ast}( A \cap E^{\textbf{C}})$. Thus, we must have $\mu^{\ast}( A )$ $\geq$ $\mu^{\ast}( A \cap E)$ +    $\mu^{\ast}( A \cap E^{\textbf{C}})$. Hence, $\mu^{\ast}( A )$ = $\mu^{\ast}( A \cap E)$ +    $\mu^{\ast}( A \cap E^{\textbf{C}})$ and $E$ is measurable.

\end{proof}

\section{Conclusion}
 Theorems $\ref{T:1}$ and $\ref{T:2}$ show that the measure space ($\widetilde{\mu}$, $ \widetilde{\mathbf{S}}$) agrees with the Caratheodory Extension when $\mu$ is a finite measure. Moreover, theorem $\ref{T:2}$ show that measurable sets are exactly limit points of $\mu$-Cauchy sequences. The $\sigma$-finite case follows from the finite case.

     .
\section{Acknowledgement}
 The first author would like to thank his grandfather Waichi Tanaka for his inspiration and financial assistance and Andrew Aames for encouraging him to progress through the graduate program. With the kind support of both, the first author has progressed further than he ever thought possible. In addition, both authors would like to thank professor Michel L. Lapidus, and professor James D. Stafney for their professional advice on this paper, and to Anne Hansen for her editing assistance.

\end{document}